\documentclass{amsart}

\usepackage[colorlinks=true,urlcolor=blue,bookmarks=true,bookmarksopen=true,citecolor=blue,pdftex]{hyperref}
\usepackage{graphicx}
\usepackage{epstopdf}
\usepackage{enumerate}

\usepackage{epsfig}
\usepackage{amscd}
\usepackage{amssymb}
\usepackage{amsxtra}
\usepackage{amsmath}
\usepackage{enumerate}
\usepackage{mathrsfs}

\usepackage{color}

\usepackage{amsmath,amsfonts,amssymb}
\usepackage{verbatim}
\usepackage[hmargin = 4cm,vmargin = 2.5cm]{geometry}
\usepackage{epsfig}
\usepackage{amsthm}

\theoremstyle{plain}

\newtheorem{thm}{Theorem}

\newtheorem{cor}[thm]{Corollary}

\newtheorem{lem}[thm]{Lemma}
\newtheorem{prop}[thm]{Proposition}

\theoremstyle{definition}

\theoremstyle{remark}

\newtheorem*{remnonum}{Remark}

\theoremstyle{plain}

\newcommand{\Id}{{{\mathchoice {\rm 1\mskip-4mu l} {\rm 1\mskip-4mu l}
      {\rm 1\mskip-4.5mu l} {\rm 1\mskip-5mu l}}}}
\renewcommand\ker{\operatorname{ker}}

\newcommand\Area{\operatorname{Area}}

\def\ZZ{\mathbb{Z}}

\def\RR{\mathbb{R}}

\def\dd{\text{d}}
\def\Ham{Ham}
\def\Cal{Cal}

\def\traj{\operatorname{traj}}
\def\genus{genus}

\begin{document}

\bibliographystyle{alphanum}

\title{Hofer's length spectrum of symplectic surfaces}
\date{\today} 
\author{Michael Khanevsky}
\address{Michael Khanevsky, Dept. of Mathematics, 
University of Chicago, 5734 S. University Avenue, Chicago, IL 60637, USA}
\email{khanev@math.uchicago.edu}
\thanks{The author was supported by NSF grant DMS-1105813 and also a Simons instructorship.}

\begin{abstract}
Following a question of F. Le Roux, we consider a system of invariants $l_A : H_1 (M) \to \RR$ of a symplectic surface $M$.
These invariants compute the minimal Hofer energy needed to translate a disk of area $A$ along a given homology class and
can be seen as a symplectic analogue of the Riemannian length spectrum.
When M has genus zero we also construct Hofer- and $C_0$-continuous quasimorphisms $\Ham(M)\to H_1(M)$ that compute trajectories
of periodic nondisplaceable disks.
\end{abstract}

\maketitle

\section{Introduction and results}

Let $(M, \omega)$ be an open symplectic surface of finite type, possibly with boundary. Pick an open disk $\mathcal{D} \subset M$ of 
$\Area(\mathcal{D}) = \int_\mathcal{D} \omega = A$
and a compactly supported Hamiltonian $\phi \in \Ham(M)$ such that $\phi (\mathcal{D}) = \mathcal{D}$.
We define the \emph{trajectory} $\traj_{\mathcal{D}} (\phi) \in \pi_1(M)$ in the following way.
Pick a Hamiltonian isotopy $\phi_t$ which connects the identity to $\phi$. Roughly speaking, 
$\traj_{\mathcal{D}} (\phi)$ is the free homotopy class of the trajectory of $\mathcal{D}$ under $\phi_t$. 
More formally, pick $x \in \mathcal{D}$ and denote by $\gamma \subset \mathcal{D}$ a curve which connects $\phi (x)$ with $x$.
We catenate the trajectory of $x$ under $\phi_t$ with $\gamma$:
\[
  \traj_{\mathcal{D}} (\phi) = [\{\phi_t (x)\}_t * \gamma] \in \pi_1 (M).
\]
It is not difficult to see that $\traj_{\mathcal{D}} (\phi)$ does not depend on the choice of $x$, $\gamma$ and $\phi_t$.

Denote by $\| \cdot \|$ the Hofer norm on the group $\Ham(M)$: 
\[
	\| \phi \| = \inf \int_0^1 \max_{p \in M} H (p, t) - \min_{p \in M} H (p, t) \mathrm{d}t ,
\]
where the infimum goes over all compactly supported Hamiltonians $H: M \times [0, 1] \to \RR$ 
such that $\phi$ is the time-$1$ map of the corresponding flow.

\medskip

Following F. Le Roux we consider the following invariant. Given $\alpha \in \pi_1 (M)$ and a disk $\mathcal{D}$ of area 
$0 < A < \Area(M)$ we define
\[
  l_A (\alpha) = \inf \{ \|\phi\| \; \big| \; \phi \in \Ham(M) \; \text{s.t.} \; \phi (\mathcal{D}) = \mathcal{D} \; \text{and} \; 
			\traj_{\mathcal{D}} (\phi) = \alpha \}.
\]
Since all disks in $M$ of the same area are equivalent, $l_A$ depends only on the area $A$ and not on a choice of particular $\mathcal{D}$. 

Clearly, $l_A$ satisfies the triangle inequality: let $a, b$ be two based loops with the same base point, denote $c = a * b$.
Then $l_A ([c]) \leq l_A ([a]) + l_A ([b])$. Using the argument from ~\cite{La-MD:L-infty-geom}, 
one shows that $l_A (\alpha) \geq A$ for all $\alpha \neq 0$. Namely, given $\phi \in \Ham(M)$ with $\traj_{\mathcal{D}} (\phi) \neq 0$, 
the lift of $\phi$ to the universal cover displaces 
a lift of $\mathcal{D}$, hence $\| \phi \| \geq A$ by the energy-capacity inequality. Therefore $l_A$ is a nontrivial system of invariants
which behaves as a symplectic analog of the length spectrum.

Question 5 from ~\cite{LR:6-questions} addresses a computation of $l_A$ in the case when $M$ is an annulus and $\mathcal{D}$ is nondisplaceable.
Its solution appeared in ~\cite{Kh:disk-an}, this paper extends the results to general surfaces.

\begin{remnonum}
One may change the definition of $l_A$ and restrict attention to those $\phi$ that fix $\mathcal{D}$ pointwise. All estimates and results
presented in this article remain true also for this setting, the same ideas apply after minor adjustments of the argument.
\end{remnonum}

\medskip

In this article we discuss a weaker version of this invariant where we consider trajectories in a class $\alpha \in H_1 (M; \ZZ)$. 
We prove the following:
\begin{thm} \label{T:l_bound} given $0 < A < \Area(M)$ and $\alpha \in H_1 (M; \ZZ)$, the following estimates hold:
  \begin{itemize}
	 \item
		If $A \leq \Area(M) / 2$, then $l_A (\alpha) \leq 2 A$.
	 \item
		If $\genus(M) > 0$, then $l_A (\alpha) \leq 2 A$.
	 \item
		Otherwise ($\genus(M) = 0$ and $A > \Area (M)/2$) $l_A : H_1 (M; \ZZ) \to \RR$ is comparable to a homogeneous norm on $H_1 (M; \ZZ)$.
		It is not important which norm we consider: $H_1 (M; \RR)$ is finitely generated, hence all homogeneous norms on $H_1 (M; \ZZ)$ are equivalent.
  \end{itemize}
\end{thm}

Certain partial results are available also for the homotopical version $l_A :  \pi_1 (M) \to \RR$, see the discussion in Section~\ref{S:disc_homo}.

\medskip

We also show:
\begin{thm} \label{T:quas}
  Let $(M, \omega)$ be a disk with $k$ punctures, pick $A \in \left(\frac{1}{2}\Area (M), \Area (M)\right)$. 
  There exists a family $P_A = \{\rho_A\}$ of invariants $\rho_A : \Ham (M) \to H_1 (M; \RR)$
  that satisfy:
  \begin{itemize}
	\item
	  $\rho_A$ is a homogeneous quasimorphism,
	\item
	  $\rho_A$ is Hofer-Lipschitz and $C_0$-continuous,
	\item
	  Suppose that $\phi \in \Ham(M)$ has an invariant disk $D$ of $\Area(D) \geq A$. Then 
	  $\rho_A (\phi) = [\traj_{D} (\phi)] \in H_1 (M; \ZZ) \subset H_1 (M; \RR)$.
  \end{itemize}
  $P_A$ contains continuum of such quasimorphisms that are linearly independent.
\end{thm}
These $\rho_A$ can be seen as one of many ways to generalize the rotation number to dimension two.
The Lipschitz property $|\rho_A (\phi)| \leq c \| \phi \|$ implies the lower bound for the third case of Theorem~\ref{T:l_bound}:
\begin{eqnarray*}
 l_A (\alpha) &=& \inf \{ \|\phi\| \; \big| \; \phi (\mathcal{D}) = \mathcal{D} \; \text{and} \; 
 			[\traj_{\mathcal{D}} (\phi)] = \alpha)\} \geq \\
 			& \geq & \inf \left\{ \frac{|\rho_A(\phi)|}{c} \; \big| \; \phi (\mathcal{D}) = \mathcal{D} \; \text{and} \; 
 			[\traj_{\mathcal{D}} (\phi)] = \alpha)\right\} = \frac{|\alpha|}{c}.
\end{eqnarray*}
Additional properties and applications of invariants $\rho_A$ are discussed in Section~\ref{S:disc_quas}.

\bigskip

This paper is organized as follows. In order to show the upper bounds for $l_A$ as in Theorem~\ref{T:l_bound} we
describe explicit Hamiltonian isotopies that achieve the desired energy bounds. Details are provided in Section~\ref{S:up}.
In Section~\ref{S:qm} we construct $\rho_A$ for Theorem~\ref{T:quas} from the Calabi quasimorphism $\Ham(S^2) \to \RR$ described in ~\cite{En-Po:calqm}.
(Note that the conditions on $M$ in Theorem~\ref{T:quas} are equivalent to the statement that $M$ embeds to a sphere.)
Section~\ref{S:disc} discusses possible 
generalizations of $l_A$ and provides few applications of the quasimorphisms from Theorem~\ref{T:quas}.

\medskip

\emph{Acknowledgements:}
The author would like to thank D. Calegari, V. Humiliere, F. Le Roux and F. Zapolsky for useful discussions and comments.

\section{Upper bounds} \label{S:up}

In this section we show upper bounds for Theorem~\ref{T:l_bound}. We construct explicit Hamiltonian isotopies that translate a disk
along a given homology class and remain within the prescribed energy bound.

\subsection{Suppose $A = \Area{\mathcal{D}} \leq \Area (M) / 2$} \label{S:up-disp}
We assume first that $A < \Area (M) / 2$.

Consider the following Hamiltonian in $(\RR^2, \mathrm{d}x \wedge \mathrm{d}y)$. We pick two disjoint disks $D_1$ and $D_2$ of area $A$ 
in the plane, connect them by two narrow nonintersecting strips (`pipes') as described in Figure~\ref{F:swap1}. Let $H$ be an autonomous 
Hamiltonian which equals one in the area bounded by the two disks and the pipes, zero outside and linearly interpolated in between.
We apply to $H$ a $C_0$-small smoothing near the singular points.

\begin{figure}[!htbp]
\begin{center}
\includegraphics[width=0.7\textwidth]{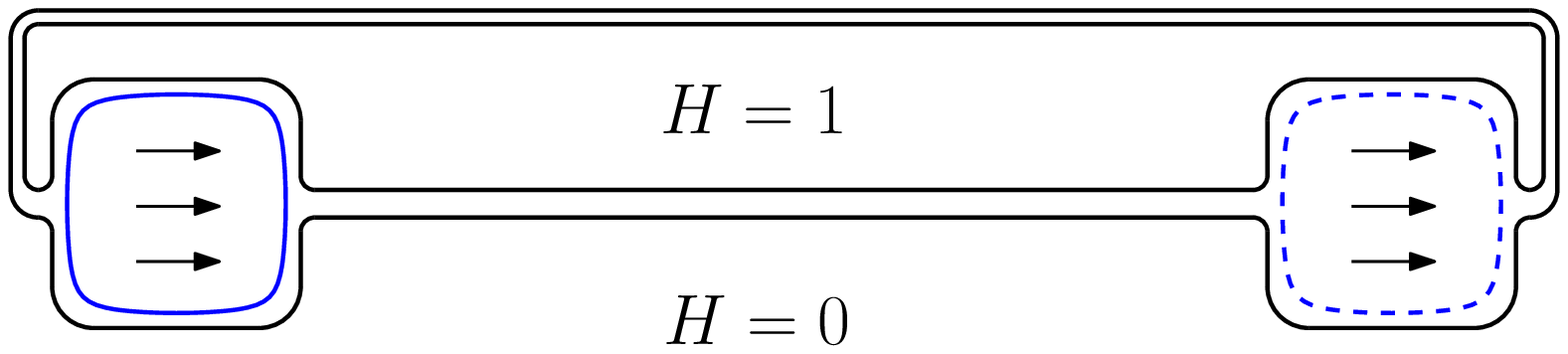}
\caption{}
\label{F:swap1}
\end{center}
\end{figure}

The flux of the Hamiltonian flow through a curve $\gamma$ equals to the difference of values of $H$ at the endpoints of 
$\gamma$. Hence if we choose $\gamma$ to be a cut going to the inner region (where $H = 1$) from the outside, the flux through $\gamma$ is one.
This implies that the resulting flow is relatively slow inside the disks and accelerates in the pipes where such a cut $\gamma$ can be very short. 
For appropriate choice of the
smoothing of $H$ the time-($A+\varepsilon$) map of the flow of $H$ will move $D_1$ to $D_2$. (In fact, it is possible to arrange that
the two disks will be swapped.) The Hofer norm is at most $\int_0^{A + \varepsilon} (\max H - \min H) \mathrm{d}t = A + \varepsilon$. 
$\varepsilon$ depends on the area of the pipes and the choice of smoothing of $H$ and can be made arbitrarily small.
Note that the energy cost of $(A+\varepsilon)$ is optimal since every Hamiltonian that displaces $D_1$ has energy 
$\geq A$ by the energy-capacity inequality.

We augment this construction as in Figure~\ref{F:swap2}: we choose the first pipe to go along an interval connecting $D_1$ with $D_2$ and
let the second pipe lie in a small neighborhood of the two disks and the first pipe.

\begin{figure}[!htbp]
\begin{center}
\includegraphics[width=0.65\textwidth]{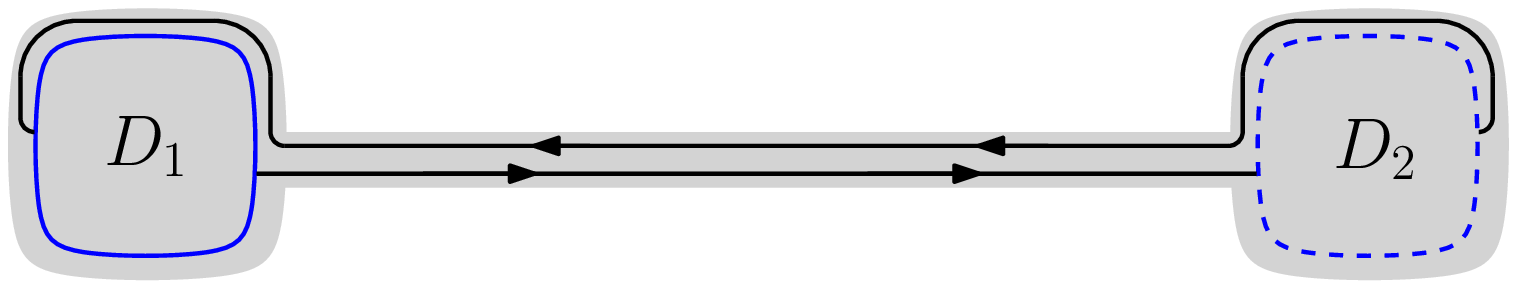}
\caption{}
\label{F:swap2}
\end{center}
\end{figure}

This way the Hamiltonian $H$ is supported in a small neighborhood of the two disks and an interval. This construction can be
embedded to any surface given two disjoint copies of a disk and a simple path connecting boundary points of the disks. This allows 
us to translate a disk along a simple path.

\medskip

Let $D_1, D_2$ be two disjoint copies of $\mathcal{D}$ in $M$ and let $\gamma_1$ be a simple path connecting $D_1$ to $D_2$
and $\gamma_2$ be a simple path connecting $D_2$ to $D_1$. $\gamma_1$ and $\gamma_2$ may intersect one another but not the disks. We apply the construction above 
in order to move $D_1$ to $D_2$ along $\gamma_1$ and then move back to $D_1$ along $\gamma_2$. The total energy cost of such process
equals $2 \Area(\mathcal{D}) + 2 \varepsilon$ and the result is a translation of the disk $D_1$ along a trajectory in class
$[\gamma_1 * \gamma_2]$. (By abuse of notation we extend $\gamma_1, \gamma_2$ inside the disks so that they have the same endpoints
and catenation makes sense.)

\begin{figure}[!htbp]
\begin{center}
\includegraphics[width=0.7\textwidth]{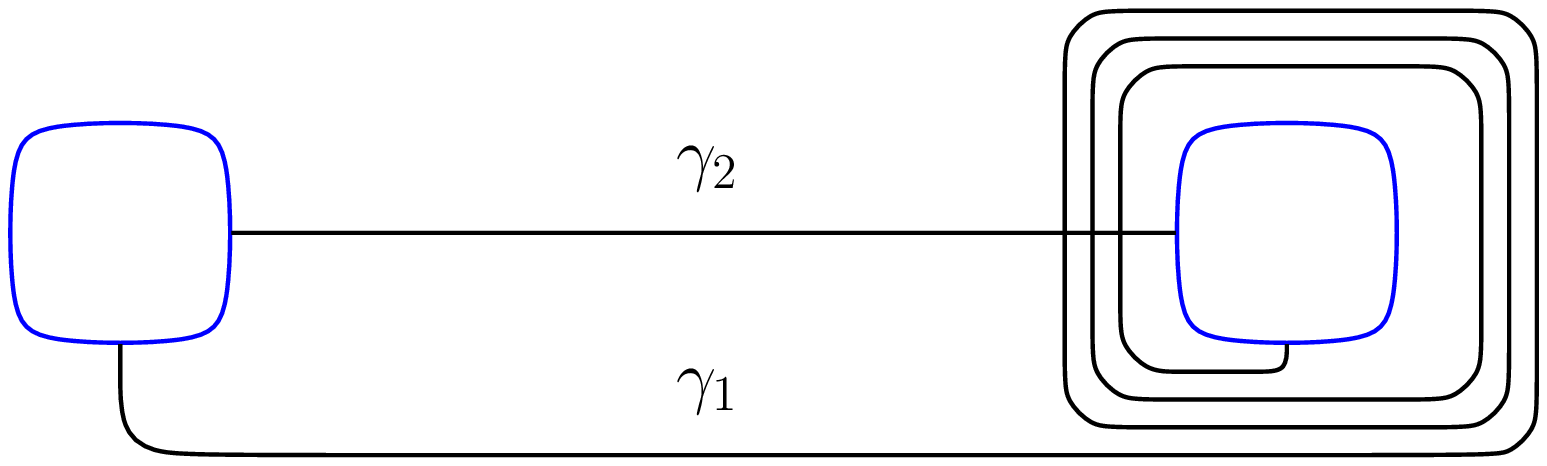}
\caption{}
\end{center}
\end{figure}

The upper bound follows from the following topological lemma:
\begin{lem} \label{l:loop2path}
  Let $M$ be a surface and $\alpha \in \pi_1 (M)$. Then there exist two simple curves $\gamma_1$, $\gamma_2$ with the same 
  endpoints such that $\alpha = [\gamma_1 * -\gamma_2]$. ($-\gamma_2$ denotes $\gamma_2$ with reversed orientation.)
\end{lem}

Indeed, pick a class $\alpha \in \pi_1 (M)$. By this lemma there exist two points $p, q \in M$ and two simple curves 
$\gamma_1, \gamma_2$ connecting them so that $[\gamma_1 * \gamma_2] = \alpha$. We replace $p$ and $q$ by small disks
$D_1, D_2$ which intersect both curves only in small intervals near the endpoints. Pick a symplectic form $\omega'$ on $M$
so that both disks have area $A$ and the total area is $\int_M \omega' = \int_M \omega$. $(M, \omega')$ is symplectomorphic to
$(M, \omega)$ by Moser's argument. We construct a Hamiltonian 
in $(M, \omega')$: translate $D_1$ along $\gamma_1$ to $D_2$ and then back along $\gamma_2$. The resulting deformation $\phi$ satisfies
\[
  \traj_{D_1} (\phi) = [\gamma_1 * \gamma_2] = \alpha \in \pi_1 (M)
\]
at the cost of $2 \Area(\mathcal{D}) + 2 \varepsilon$, where $\varepsilon$ is arbitrarily small.

\medskip

Assume that $A = \Area (M) / 2$ and fix $\alpha = [\gamma_1 * \gamma_2] \in \pi_1(M)$.
Pick a smaller disk $D \subset \mathcal{D}$, suppose that $\Area (D) = A - \epsilon$.
We apply the previous argument and construct $\phi \in \Ham(M)$ which translates $D$ along $\alpha$.
Then we deform $\phi (\mathcal{D})$ back to $\mathcal{D}$ by another Hamiltonian isotopy $f$ which fixes the smaller disk $D$.
Note that the annulus $\phi (\mathcal{D} \setminus D)$ is spread along $\gamma_1 \cup \gamma_2$ and its shape does not depend on the choice of $D$. 
This implies that $f$ can be chosen with $\|f\| \leq c \epsilon$ where $c$ depends on $\gamma_1, \gamma_2$ but not on $\epsilon$.
Therefore 
\[
  l_A (\alpha) \leq \| \phi \| + \| f \| \leq 2A + 2 \varepsilon + c \epsilon.
\]
The result follows by letting $\varepsilon, \epsilon \to 0$.

\begin{remnonum}
  Note that we proved the upper bound for $\alpha \in \pi_1 (M)$ which, in particular, implies the same bound for the homological spectrum.
\end{remnonum}

\medskip

\begin{proof}[Proof of lemma]
  Pick $\alpha \in \pi_1 (M)$. Let $\gamma$ be a loop representing $\alpha$ which has finite number of 
  transverse intersections. We deform $\gamma$ so that it can be cut into a pair of simple curves. Pick arbitrary point $p \in \gamma$
  which will be a starting point of $\gamma_1$. We draw $\gamma_1$ by traversing $\gamma$ starting from $p$. We prevent self-intersections 
  of $\gamma_1$ as follows. Once we arrive to a self-crossing we stretch the previously drawn path an push it in a 
  small neighborhood in front of
  the pen (see Figure~\ref{F:pathlemma}). After passing several crossing we will have a number of strings pushed in this manner.
  During of this process we construct a simple path which is homotopic to the corresponding 
  arc of $\gamma$ relative the endpoints. We stop the process when we arrive to a point $q \in \gamma$ such that the arc
  $[q, p] \subset \gamma$ is simple. Set $\gamma_2 = [q, p] \subset \gamma$. 
  Lemma follows since $\gamma_1$ is homotopic to $[p, q] \subset \gamma$ relative the endpoints.
  
\begin{figure}[!htbp]
\begin{center}
\includegraphics[width=\textwidth]{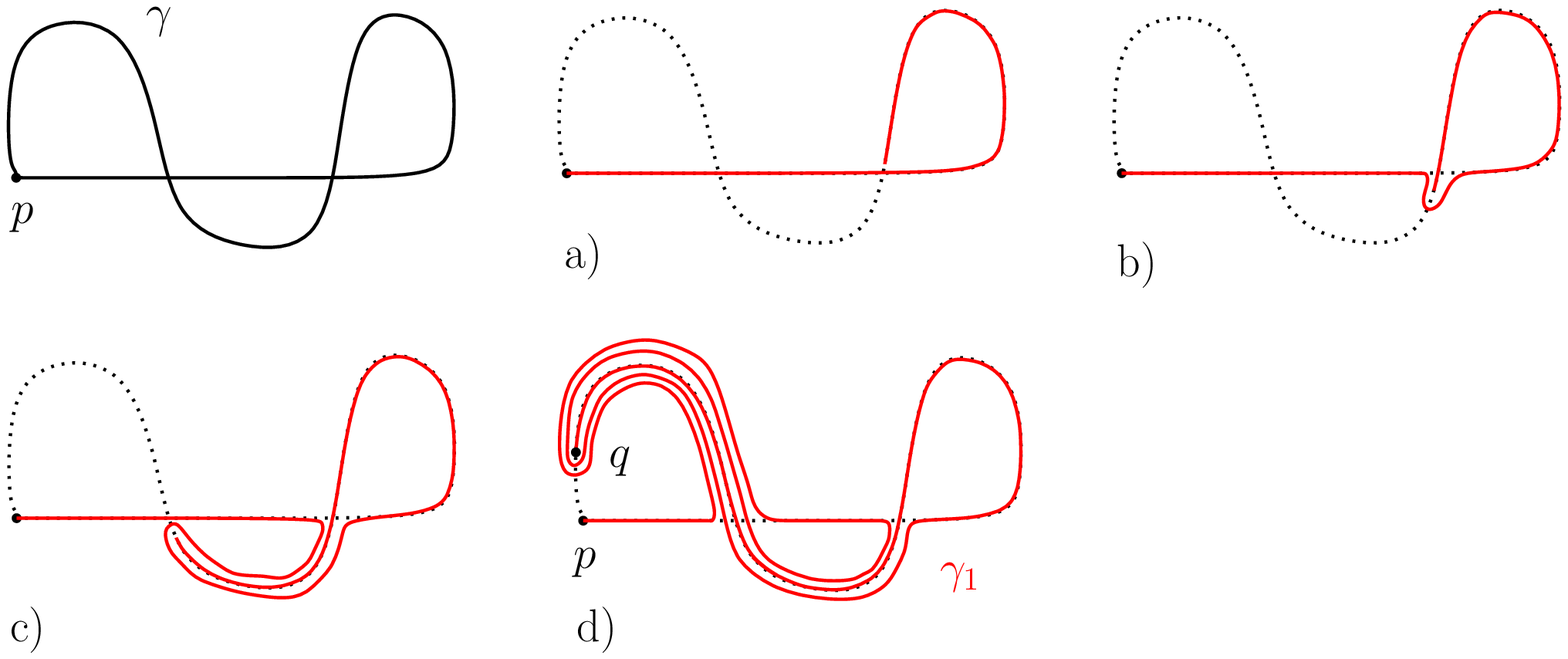}
\caption{}
\label{F:pathlemma}
\end{center}
\end{figure}

\end{proof}

\subsection{Arbitrary $\mathcal{D}$}. 
We show first that one can translate a disk along a simple loop using energy $\Area (\mathcal{D}) + \varepsilon$. 
Then we consider separately the cases $\genus(M) = 0$ and $\genus (M) > 0$.

Consider the following Hamiltonian in $(\RR^2, \mathrm{d}x \wedge \mathrm{d}y)$. We pick a disk $D \subset \RR^2$ of area $A$ 
and connect two boundary points of the disk by a circular pipe. Let $H$ be an autonomous 
Hamiltonian which equals one in the inner region, zero outside and linearly interpolated in between.
We apply to $H$ a $C_0$-small smoothing near the singular points (see Figure~\ref{F:loop}).

\begin{figure}[!htbp]
\begin{center}
\includegraphics[width=0.4\textwidth]{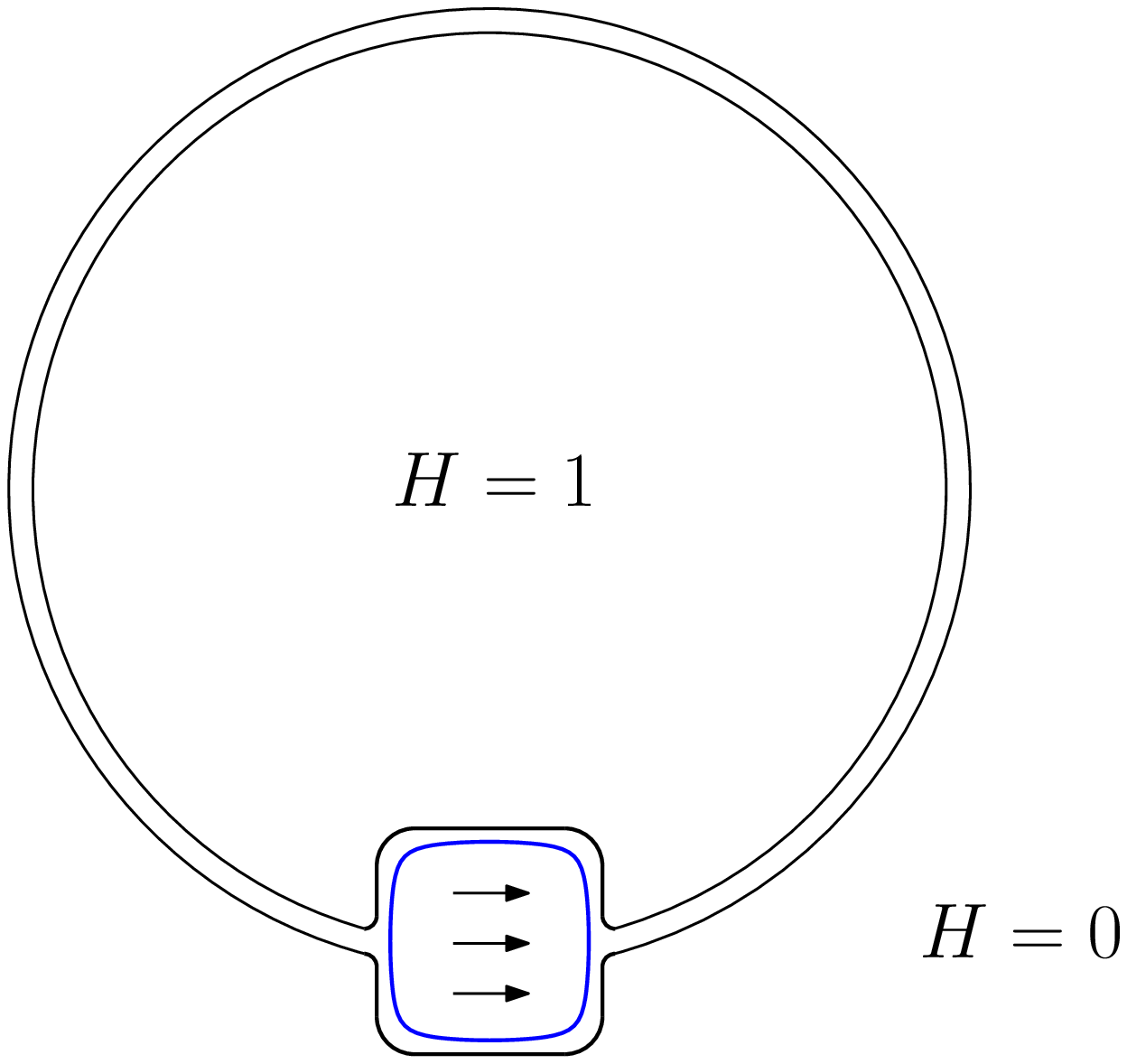}
\caption{}
\label{F:loop}
\end{center}
\end{figure}

The Hamiltonian flow of $H$ is relatively slow inside the disk $D$ and fast within the pipe. For an appropriate choice of a
smoothing of $H$ there exists $\varepsilon > 0$ such that $D$ is a fixed set of the time-$(A+\varepsilon)$ map of the flow of $H$.
The energy is at most $A+\varepsilon$ and $\varepsilon$ depends on the area of the pipe and a choice of smoothing and can be made arbitrarily small.

We cut $H$ off inside the circle so that it is supported in a small neighborhood of the disk and the pipe. Note that this cutoff
does not affect the flow of $H$ in the disk $D$ and the pipe. This construction can be copied to any symplectic surface
given a disk and a simple curve which connects boundary points of the disk and does not intersect it away from the endpoints.
This implies the following result. 
\begin{lem}
  Suppose that $\alpha \neq 0$ is represented by a simple loop. Then $l_A (\alpha) = A$.
\end{lem}
\begin{proof}
  Let $\gamma$ be a simple loop representing $\alpha$. As in the previous subsection, we deform
  $\omega$ so that there exists a disk of area $A$ whose boundary points are connected by an arc of $\gamma$ and this arc does
  not intersect the disk except for the endpoints. An application of the Hamiltonian described above implies $l_A (\alpha) \leq A + \varepsilon$.
  The lemma follows by letting $\varepsilon \to 0$ and from the fact that $l_A (\alpha) \geq A$.
\end{proof}

\begin{cor}\label{C:w_length}
  Let $S \subseteq H_1 (M; \ZZ)$ be the set of classes represented by simple loops. It is easy to see that $S$ generates $H_1 (M; \ZZ)$.
  Denote by $\| \cdot \|_S$ the norm on $H_1 (M; \ZZ)$ given by word length with respect to $S$. The triangle inequality implies
  $l_A (\alpha) \leq A \cdot \| \alpha \|_S$. The same argument holds also for the homotopical length spectrum. 
\end{cor}

Given any surface $M$,  $H_1 (M; \ZZ)$ admits a basis $B$ represented by simple loops. By the triangle inequality, 
$l_A (\alpha) \leq A \cdot \| \alpha \|_B$ where $\| \cdot \|_B$ is the word length norm with respect to $B$.
$\| \cdot \|_B$ is a homogeneous norm on $H_1(M; \ZZ)$. This shows the upper bound for the third case in Theorem~\ref{T:l_bound}.
A more careful argument may provide better constants. For example, in the case of an annulus $M$ the bound can be improved to 
\[ 
  l_A (n \cdot \text{generator}) \leq (2 A - \Area(M)) \cdot |n| + \Area(M)
\]
(see ~\cite{Kh:disk-an}).

\medskip

Suppose that $\genus(M) > 0$. The upper bound for Theorem~\ref{T:l_bound} follows from the lemma below and Corollary~\ref{C:w_length}.

\begin{lem}
  Let $M$ be a surface with positive genus. Then every homology class in $H_1 (M; \ZZ)$ can be represented by a sum of two simple loops.
\end{lem}
\begin{proof}
  \underline{Step 1}: Suppose $M = S^1 \times S^1$ is a torus. $\alpha = [\{pt\} \times S^1]$ and $\beta = [S^1 \times \{pt\}]$
  generate $H_1 (M; \ZZ) = \ZZ<\alpha, \beta>$. Any class $n \alpha + m \beta$ can be decomposed into the sum 
  \[
	n \alpha + m \beta = \big((n -1)\alpha + \beta\big) + \big( \alpha + (m-1) \beta \big)
  \]
  where each summand is represented by a simple loop (see Figure~\ref{F:torus}).
  \begin{figure}[!htbp]
  \begin{center}
  \includegraphics[width=0.7\textwidth]{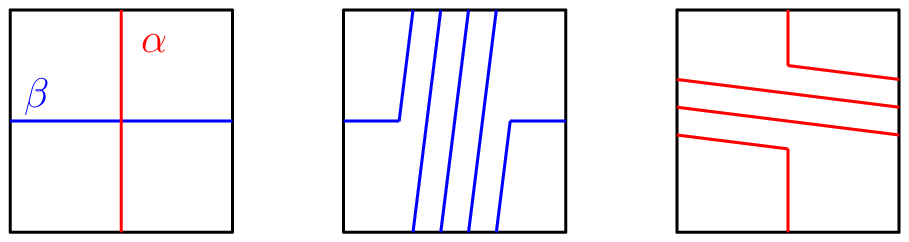}
  \caption{}
  \label{F:torus}
  \end{center}
  \end{figure}

  \medskip
  
  \underline{Step 2}: Suppose $M = S^1 \times S^1 \setminus \{p_1, \ldots, p_k\}$ is a punctured torus. 
  Without loss of generality we assume that all punctures lie on a horizontal meridian. $H_1 (M; \ZZ)$ admits a basis
  of $k$ vertical meridians $\alpha_1, \ldots, \alpha_k$ and a horizontal meridian $\beta$ (see Figure~\ref{F:ptorus}).
  Then every 
  \[
	\alpha = m_1 \alpha_1 + \ldots + m_k \alpha_k + n \beta \in H_1 (M; \ZZ)
  \]
  admits a decomposition
  \[
	\alpha = \big((m_1-1) \alpha_1 + \ldots + m_k \alpha_k + \beta\big) + \big(\alpha_1 + (m-1) \beta \big)
  \]
  where each summand is represented by a simple loop.
  
  \begin{figure}[!htbp]
  \begin{center}
  \includegraphics[width=0.72\textwidth]{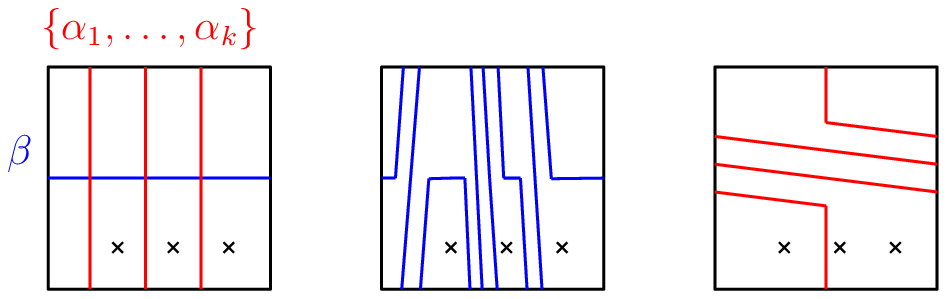}
  \caption{}
  \label{F:ptorus}
  \end{center}
  \end{figure}

  \medskip
  
  \underline{Step 3}: Suppose $\genus(M) > 1$. $M$ can be decomposed into a connected sum of punctured tori 
  \[
	M = T_1 \# \ldots \# T_l, \quad H_1 (M; \ZZ) = H_1 (T_1; \ZZ) + \ldots + H_1 (T_l; \ZZ).
  \]
  We identify $H_1 (T_i; \ZZ)$ with their image in $H_1 (M; \ZZ)$ by abuse of notation.
  \begin{figure}[!htbp]
  \begin{center}
  \includegraphics[width=0.5\textwidth,natwidth=693,natheight=426]{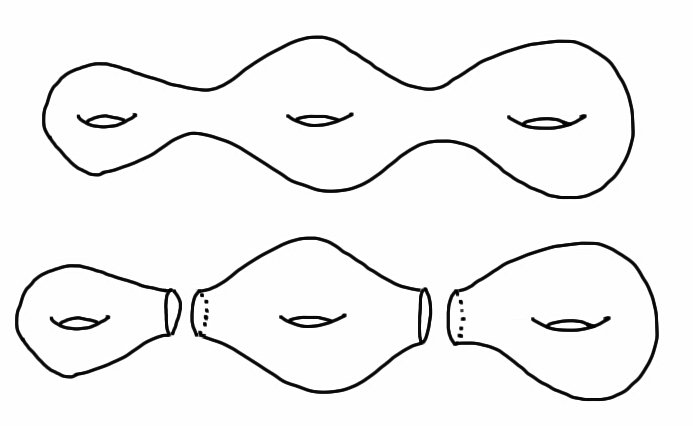}
  \label{F:cut}
  \end{center}
  \end{figure}
  
  Pick $\alpha \in H_1 (M; \ZZ)$. There exist $\alpha_i \in H_1 (T_i; \ZZ)$ such that $\alpha = \sum \alpha_i$.
  By the previous step we choose homology classes $\beta_i, \beta'_i \in H_1 (T_i; \ZZ)$
  represented by simple loops such that $\alpha_i = \beta_i + \beta'_i$.
  Now
  \[
	\alpha = \sum \alpha_i = \sum (\beta_i + \beta'_i) = \sum \beta_i + \sum \beta'_i.
  \]
  Both $\sum \beta_i$ and $\sum \beta'_i$ are represented by simple loops: consider a connected sum of the corresponding simple loops in each $T_i$.
  The result is a simple loop in $M$ since $T_i$ are disjoint.
\end{proof}

\section{Construction of quasimorphisms $\rho_A$}\label{S:qm}

Let $G$ be a group. A function $r : G \to \RR$ is called a \emph{quasimorphism} if there exists 
a constant $R$ such that $|r(fg) - r(f) - r(g)| \leq R$ for all $f, g \in G$. 
$R$ is called the \emph{defect} of $r$. The quasimorphism $r$ 
is \emph{homogeneous} if it satisfies $r(g^m) = mr(g)$ for all $g \in G$ and $m \in \ZZ$.
Any homogeneous quasimorphism satisfies $r(fg) = r(f) + r(g)$ for commuting elements $f, g$ and is invariant under conjugations.

\medskip

We construct quasimorphisms $\rho_A$ described in Theorem~\ref{T:quas}. Recall that $M$ is a surface of finite type of genus zero.
$M$ has finite area, so without loss of generality we rescale $\omega$ to get $\Area(M) = 1$. Such $M$ is symplectomorphic to the interior of 
a $k$-times punctured disk $U \setminus \{p_1, \ldots , p_k\}$ where $k$ is the rank of $H_1 (M; \ZZ)$ and $\Area (U) = 1$. 

\begin{prop} \label{P:quas}
  Let $(\mathbb{A}, \omega)$ be an annulus (disk with one puncture). Suppose that $\Area(\mathbb{A}) = 1$. 
  Pick $\frac{1}{2} < A < 1$ and choose a generator $\alpha$ for $H_1 (M, \ZZ)$.
  There exists a family $P_A = \{\rho_A\}$ of invariants $\rho_A : \Ham (\mathbb{A}) \to \RR$
  that satisfy:
  \begin{itemize}
	\item
	  $\rho_A$ is a homogeneous quasimorphism,
	\item
	  $\rho_A$ is Hofer-Lipschitz and $C_0$-continuous,
	\item
	  Suppose that $\phi \in \Ham(\mathbb{A})$ has an invariant disk $D$ of $\Area(D) \geq A$. Then 
	  $[\traj_{D} (\phi)] = \rho_A (\phi) \cdot \alpha$.
  \end{itemize}
  $P_A$ contains continuum of such quasimorphisms that are linearly independent.
\end{prop}

The proposition implies Theorem~\ref{T:quas}.
\begin{proof}[Proof of Theorem~\ref{T:quas}]
Pick such a quasimorphism $\rho_\mathbb{A} : \Ham (\mathbb{A}) \to \RR$ and denote by 
\[
  i_j : U \setminus \{p_1, \ldots , p_k\} \hookrightarrow U \setminus \{p_j\} \simeq \mathbb{A}, \quad 1 \leq j \leq k
\]
the embeddings obtained by gluing all punctures except for $p_j$ and identifying the resulting once punctured disk with $\mathbb{A}$.
The pullback quasimorphisms $\rho_j = i_j^* \rho_\mathbb{A} : \Ham (U \setminus \{p_1, \ldots , p_k\}) \to \RR$ satisfy the Lipschitz and
continuity properties. Given a Hamiltonian $\phi$ with a fixed disk $\mathcal{D}$ of area $\geq A$, $\rho_j (\phi)$ counts
the winding number of $\traj_{\mathcal{D}} (\phi)$ around $p_j$. Denote by $\gamma_j$ a small positively oriented circle in $U$ around $p_j$
which does not encircle other punctures. $\{[\gamma_j]\}_{j=1}^k$ is a basis for $H_1(U \setminus \{p_1, \ldots , p_k\}; \ZZ)$.
The dual basis $\{[\gamma_j]^*\}$ allows us to write
\[
  \rho_j (\phi) = [\gamma_j]^*[\traj_{\mathcal{D}} (\phi)]
\]
for all Hamiltonians $\phi$ with a fixed disk $\mathcal{D}$ as above. This implies 
\[
  [\traj_{\mathcal{D}} (\phi)] = \sum_{j=1}^k [\gamma_j]^*[\traj_{\mathcal{D}} (\phi)] \cdot [\gamma_j] = \sum_{j=1}^k \rho_j (\phi) [\gamma_j].
\]
\[
  \rho_A = \sum_{j=1}^k \rho_j (\cdot) [\gamma_j] : \Ham (U \setminus \{p_1, \ldots , p_k\}) \to H_1 (U \setminus \{p_1, \ldots , p_k\}; \RR)
\]
is a homogeneous quasimorphism which satisfies all properties of Theorem~\ref{T:quas}. Continuum of linearly independent choices of 
$\rho_\mathbb{A}$ implies that for a choice of $\rho_A$.
\end{proof}

\medskip

Proof of Proposition~\ref{P:quas} takes the rest of this section. It can be found in ~\cite{Kh:disk-an}, we give it below 
for the sake of completeness.  
We use notation $\mathbb{A} = S^1 \times (0, 1) \simeq \RR / \ZZ \times (0, 1)$ equipped with the 
standard symplectic form $\omega = \mathrm{d}\theta \wedge \mathrm{d}h$, $\Area (\mathbb{A}) = 1$. Without loss of generality we assume
that the generator $\alpha \in H_1 (\mathbb{A}; \ZZ)$ is represented by the positively oriented circle $S^1 \times \{0.5\}$.
Let $\mathcal{D} \subset \mathbb{A}$ be a disk with $\Area (\mathcal{D}) = A$, denote $L = \partial \mathcal{D}$. 
Put $S = \{ \phi \in Ham (\mathbb{A}) \, | \, \phi (\mathcal{D}) = \mathcal{D}\}$ to be the stabilizer of $\mathcal{D}$ in 
$\Ham(\mathbb{A})$. We denote $S_n = \{\phi \in S \, | \, [\traj_\mathcal{D} (\phi)] = n \alpha\}$.

Pick a function $\widehat{H} : \mathbb{A} \to \RR$ with compact support such that $\widehat{H}(\theta, h) = h$
away from a small neighborhood of $\partial \mathbb{A}$. Denote by $\widehat{\Phi}$ the time-$1$ map of the 
Hamiltonian flow generated by $\widehat{H}$. It is easy to see that $\widehat{\Phi} \in S$ with trajectory 
$[\traj_\mathcal{D} (\widehat{\Phi})] = \alpha$. Note that $S = \bigcup_{n \in \ZZ} S_n = \bigcup_{n \in \ZZ} \widehat{\Phi}^n S_0$.
We construct a quasimorphism $\rho_A$.

\underline{Step I:} it is sufficient to construct a homogeneous quasimorphism $\rho: Ham(\mathbb{A}) \to \RR$ which satisfies:
\begin{itemize}
  \item 
	$\rho$ is Hofer-Lipschitz,
  \item
	$\rho (S_0) = 0$, $\rho (\widehat{\Phi}) = c > 0$,
  \item
	$\rho$ vanishes on $\Ham(U)$ for all disks $U \subset \mathbb{A}$.
\end{itemize}

\begin{proof} $\rho$ vanishes on Hamiltonians supported in a disk. It follows from the results of ~\cite{En-Po-Py:qm-continuity} that 
such $\rho$ is continuous in the $C^0$-topology.
Any $\phi \in S_n$ decomposes as 
$\phi = \widehat{\Phi}^n \circ s$ for some $s \in S_0$. Hence 
$$|\rho (\phi) - c n| = |\rho (\widehat{\Phi}^n \circ s) - n \rho(\widehat{\Phi}) - 0| = 
|\rho (\widehat{\Phi}^n \circ s) - \rho(\widehat{\Phi}^n) - \rho(s)| < R.$$
($R$ denotes the defect of $\rho$). It follows that 
\[
	\rho (\phi) = c \cdot n + \delta_{\phi} 
\]
where $|\delta_{\phi}| \leq R$. Note that $\phi^k \in S_{nk}$, hence by homogenuity
\[
	\rho (\phi) = \frac{\rho (\phi^k)}{k} = \frac{c \cdot n k + \delta_{\phi^k}}{k} = c \cdot n + \frac{\delta_{\phi^k}}{k}
\]
which implies $\rho (\phi) = c \cdot n$ in the limit as $k \to \infty$.

We set $\rho_A = \frac{\rho}{c}$. Suppose that $\phi \in \Ham(\mathbb{A})$ has an invariant disk $D$ of $\Area(D) \geq A$.
\begin{enumerate}
  \item 
	We consider the case $D = \mathcal{D}$. Then $\phi \in S_n$ for some $n \in \ZZ$ and 
	\[
	  \rho_A (\phi) \cdot \alpha = \frac{\rho(\phi)}{c} \cdot \alpha = n \alpha =[\traj_D (\phi)].
	\]
  \item 
	The case $\Area(D) = A$. Pick $g \in \Ham(\mathbb{A})$ such that $g (D) = \mathcal{D}$. Then $g \phi g^{-1} \in S$ hence
	\[
	  \rho_A (\phi) \cdot \alpha = \rho_A (g \phi g^{-1}) \cdot \alpha =[\traj_\mathcal{D} (g \phi g^{-1})] = [\traj_D (\phi)].
	\]
  \item 
	The case $\Area(D) > A$. Pick a smaller disk $D_1 \subset D$ with $\Area(D_1) = A$. $\phi (D) = D$, choose $g \in \Ham(D)$
	such that $g (\phi (D_1)) = D_1$. Then $D_1$ is a fixed disk of area $A$ for $g \phi$ hence 
	\[
	  \rho_A (g \phi) \cdot \alpha = [\traj_{D_1} (g \phi)] = [\traj_D (\phi)].
	\]
	We claim that $\rho_A (g \phi) = \rho_A (\phi)$. Indeed, the disk $D$ is fixed for $g \phi$, $\phi$ and their iterates.
	Moreover, $(g \phi)^n$ differs from $\phi^n$ by a Hamiltonian $f_n \in \Ham(D)$.
	Hence by the homogenuity of $\rho_A$
	\[
	  |\rho_A ((g \phi)^n) - \rho_A(\phi^n) - \rho_A (f_n)| = |n \rho_A (g \phi) - n\rho_A(\phi) - 0| = n \cdot |\rho_A (g \phi) - \rho_A(\phi)|
	\]
	is bounded by the defect of $\rho_A$. Letting $n \to \infty$ we show $\rho_A (\phi) = \rho_A (g \phi)$.
\end{enumerate}
\end{proof}

\medskip

\underline{Step II:} in order to build $\rho$ we use the Calabi quasimorphism 
on $Ham (S^2)$ which was constructed by M. Entov and L. Polterovich 
in ~\cite{En-Po:calqm}. We give a brief recollection of the relevant facts.

Let $U$ be an open disk equipped with a symplectic form $\omega$. 
Let $F_t : U \to \RR$, $t \in [0, 1]$ be a time-dependent smooth function with compact support. We define
$$\widetilde{\Cal} (F_t) = \int_0^1 \left( \int_U F_t \omega \right) \dd t.$$ As $\omega$ is exact on $U$,
$\widetilde{\Cal}$ descends to a homomorphism $\Cal_U: \Ham(U) \to \RR$ which is called 
the Calabi homomorphism. Clearly, for $\phi \in \Ham(U)$, $| \Cal_U(\phi) | \leq \Area(U) \cdot \| \phi \|$.

Let $S^2$ be a sphere equipped with a symplectic form $\omega$. Suppose $\Area (S^2) = 2 A$.
For a smooth function $F: S^2 \to \RR$ the Reeb graph $T_F$ is defined as the set of connected
components of level sets of $F$ (for a more detailed definition we refer the reader to ~\cite{En-Po:calqm}). 
For a generic Morse function $F$ this set, equipped with the topology 
induced by the projection $\pi_F: S^2 \to T_F$, is homeomorphic to a tree. 
We endow $T_F$ with a measure given by $\mu (X) = \int_{\pi_F^{-1}(X)} \omega$ 
for any $X \subseteq T_F$ with measurable $\pi_F^{-1}(X)$. $x \in T_F$ is called a \emph{median} of $T_F$
if the measure of each connected component of $T_F \setminus \{x\}$ does not exceed $A$. By ~\cite{En-Po:calqm}
a median exists and is unique.
This construction can be extended to functions $F$ such that $F \big|_{supp (F)}$ is Morse.

~\cite{En-Po:calqm} describes construction of a homogeneous quasimorphism $\Cal_{S^2} : \Ham (S^2) \to \RR$.
It has the following properties: $\Cal_{S^2}$ is Hofer-Lipschitz ($|\Cal_{S^2}(\phi)| \leq 2A \cdot \| \phi \|$).
In the case when $\phi \in \Ham(S^2)$ is supported in a disk $U$ which is displaceable 
in $S^2$, $\Cal_{S^2} (\phi) = Cal_U (\phi \big|_U)$.
Moreover, for $\phi \in \Ham(S^2)$ generated by an autonomous function $F: S^2 \to \RR$, 
$\Cal_{S^2} (\phi)$ can be computed in the following way. 
Let $x$ be the median of $T_F$ and $X = \pi_F^{-1} (x)$ be the corresponding subset of $S^2$.
Then 
\[
	\Cal_{S^2}(\phi) = \int_{S^2} F \omega - 2A \cdot F(X).
\]

Given a symplectic embedding $j : \mathbb{A} \to S^2$ into a sphere of area $2 A$, 
consider the pullback $Cal_{j} = j^* (Cal_{S^2}) : Ham(\mathbb{A}) \to \RR$. 
Namely, given $\phi \in Ham(\mathbb{A})$, extend $j_* (\phi)$ to $\tilde{\phi} \in \Ham (S^2)$ by identity 
on the complement of $j (\mathbb{A})$. Then $Cal_j (\phi) = \Cal_{S^2} (\tilde{\phi})$.
Clearly, $Cal_{j}$ is a homogeneous quasimorphism.
It has the following properties:
\begin{itemize}
	\item
		$Cal_j (\phi) = \Cal_D(\phi \big|_D)$ for any $\phi$ supported in a disk $D$ of area $A$.
		To see that note that the corresponding $\tilde{\phi} \in Ham(S^2)$ 
		is supported in a displaceable disk $j(D)$ in $S^2$.
	\item
		\[
			\left|\Cal_j(\phi)\right| = \left|\Cal_{S^2}\left(\tilde{\phi}\right)\right| \leq 
			2A \cdot \left\| \tilde{\phi} \right\|_{S^2} \leq 2A \cdot \left\| \phi \right\|_{\mathbb{A}}.
		\]
	\item for an autonomous $\phi$ generated by a compactly supported function $H : \mathbb{A} \to \RR$, 
		\[
			Cal_j (\phi) = \int_\mathbb{A} H \omega - 2A \cdot H(X)
		\]
		where $X \subseteq \mathbb{A}$ is the level set component which is sent by $j$ to the median set
		of $j_* (H)$ in $S^2$.
\end{itemize}

\medskip

Consider the embeddings $j_s : \mathbb{A} \to S^2$ ($0 \leq s \leq 2 A - 1$) into a sphere of area $2 A$ 
that are given by gluing a disk of area $s$ to $S^1 \times \{0\}$ and a disk of area
$(2 A - 1 - s)$ to $S^1 \times \{1\}$.
This construction ensures that $j_s (L) = j_s (\partial \mathcal{D})$ bisects $S^2$ 
into two displaceable disks.

We pick $0 \leq s_1 < s_2 \leq 2 A - 1$ and set 
\begin{equation}\label{eq:qm_def}
	\rho = \rho_{s_1, s_2} = Cal_{j_{s_2}} - Cal_{j_{s_1}}.
\end{equation}

\medskip

We claim that $\rho$ satisfies conditions of Step I.
Obviously, $\rho$ is a homogeneous quasimorphism on $Ham(\mathbb{A})$ which satisfies 
the Lipschitz property:
\[
	|\rho (\phi)| \leq |Cal_{j_{s_2}} (\phi)| + |Cal_{j_{s_1}} (\phi)| \leq 4 A \cdot \|\phi\|.
\]
Consider the function $\widehat{H}$ which was used to define the Hamiltonian $\widehat{\Phi}$ described 
above. It is easy to see that the ``median'' level set $X_s$ of $\widehat{H}$
which is relevant for the computation of $Cal_{j_s} (\widehat{\Phi})$ is $S^1 \times \{A-s\}$.
Hence 
\[
	Cal_{j_s} (\widehat{\Phi}) = \int_\mathbb{A} \widehat{H} \omega - 2A \cdot \widehat{H}(X_s) = 
		\int_\mathbb{A} \widehat{H} \omega - 2A \cdot (A-s).
\]
This implies
\[
	\rho(\widehat{\Phi}) = Cal_{j_{s_2}} (\widehat{\Phi}) - Cal_{j_{s_1}} (\widehat{\Phi})= 2A \cdot [- (A-s_2) + (A-s_1) ] = 2A \cdot (s_2 - s_1) > 0.
\]
Suppose that $\phi \in \Ham(U)$ for some disk $U \subset \mathbb{A}$. Then
\[
  \rho(\phi) = Cal_{j_{s_2}} (\phi) - Cal_{j_{s_1}} (\phi) = Cal_{S^2} (j_{s_2, *} (\phi)) - Cal_{S^2} (j_{s_1, *} (\phi)) = 0
\]
since $j_{s_2, *} (\phi)$ is conjugate to $j_{s_1, *} (\phi)$.

\medskip

It is left to show that $\rho$ vanishes on $S_0$. Denote by 
$S'_0$ the subgroup of $S_0$ which fixes a neighborhood of $L$ pointwise.

\begin{lem}\label{lm:fix_neigh}
	Let $q$ be a homogeneous quasimorphism which is Hofer-continuous and vanishes on $S'_0$.
	Then $q$ vanishes on $S_0$.
\end{lem}
\begin{proof}
Pick an open disk $U \subset \mathbb{A} \setminus L$. $\Ham (U) \subset S'_0$, therefore 
$q$ vanishes on $\Ham (U)$. It follows from the results of ~\cite{En-Po-Py:qm-continuity} that 
$q$ is continuous in the $C^0$-topology.

Suppose $\phi \in S_0$. This implies $\phi (L) = L$ (as a set). Applying an appropriate isotopy $\psi \in S_0$ supported near $L$, 
we may ensure that $\psi \circ \phi = \Id$ on $L$. Moreover, $\psi$ may be chosen with arbitrarily small Hofer norm. 
Further, we may find a $C^0$-small diffeomorphism $h \in Ham(\mathbb{A})$ such that
$h \circ \psi \circ \phi = \Id$ in a neighborhood of $L$. 
It follows that $h \circ \psi \circ \phi \in S'_0$ and $q (h \circ \psi \circ \phi) = 0$.
Hofer- and $C^0$-continuity of $q$ imply that $q (\phi) = 0$.
\end{proof}

Let $\phi \in S'_0$. We show that $\rho (\phi) = 0$. $\phi$ splits to a composition
$\phi = \phi_\mathcal{D} \circ \phi_P$ where $\phi_\mathcal{D}$ is supported in 
$\mathcal{D}$ and $\phi_P$ is supported in the pair of pants
$P = \mathbb{A} \setminus \mathcal{D}$. $\phi_\mathcal{D}, \phi_P$ have disjoint supports, 
therefore they commute. Hence $\rho (\phi) = \rho(\phi_\mathcal{D}) + \rho (\phi_P)$. 
Note that $\phi_\mathcal{D}$ is supported in a disk, so $\rho (\phi_\mathcal{D}) = 0$.

$\phi_P$ is Hamiltonian in $\mathbb{A}$, but after the restriction to $P$
we have just $\bar{\phi}_P = \phi_P\big|_P \in Symp_c(P)$. In the argument below we apply a sequence of 
deformations to $\phi_P$ in order to get $\phi'_P$ whose restriction $\bar{\phi}'_P \in Ham (P)$. All deformations
involved in the process preserve the value $\rho (\phi_P)$. Finally, we show that
$\rho (\phi'_P) = 0$ by explicit computation.
 
The mapping class group 
$\pi_0 (Symp_c (P))$ is isomorphic to $\ZZ^3$ and is generated by Dehn twists near the
three boundary components. For the proof of this fact we refer the reader to 
~\cite{F-M:mapping-class} where the authors show that $\pi_0 (Diff_c (P)) \simeq \ZZ^3$
and is generated by Dehn twists. Note that $\phi, \psi \in Symp_c (P)$ are isotopic in 
$Symp_c$ if and only if they are isotopic in $Diff_c$. As Dehn twists belong to
$Symp_c (P)$, the statement for $\pi_0 (Symp_c (P))$ follows. 

Denote by $T_1, T_0$ Dehn twists near $S^1 \times \{1\}$, $S^1 \times \{0\}$ 
and by $T_L$ a Dehn twist in 
$P$ near $L = \partial \mathcal{D}$. 
There are Hamiltonians $\psi_L$ in $S'_0$ with arbitrary small Hofer norm
whose restriction to $P$ realizes the Dehn twist $T_L$. For example, consider a bump function supported near $\mathcal{D}$
which has small height but is very steep in an annulus near $L$. 
$\bar{\phi}_P$ is isotopic in $Symp_c (P)$
to some $T_1^{k_1} T_0^{k_0} T_L^{k_L}$ ($k_i \in \ZZ$). 
If $k_L \neq 0$ we replace the original $\phi \in S'_0$ by $\psi_L^{-k_L} \circ \phi \in S'_0$.
As $\|\psi_L\|$ can be chosen to be arbitrarily small, by continuity of $\rho$ it is enough to show 
the desired statement for the deformed $\phi$.
After the replacement $k_L$ vanishes, hence the modified $\bar{\phi}_P \sim T_1^{k_1} T_0^{k_0}$.
Note that $\bar{\phi}_P$ is induced by a Hamiltonian $\phi \in S$. The definition of
$\traj_{\mathcal{D}} (\phi)$ implies that $k_1 \alpha = [\traj_{\mathcal{D}} (\phi)] = -k_0 \alpha$
(recall that $\alpha$ is the chosen generator of $H_1 (\mathbb{A}; \ZZ)$).
The minus sign appears because opposite orientations of the boundary components
result in the opposite directions of the corresponding Dehn twists. 
Moreover, as $\phi \in S'_0$, $[\traj_{\mathcal{D}} (\phi)] = 0$ hence $k_1 = k_0 = 0$.
Therefore the restriction $\bar{\phi}_P$ belongs to the identity component of $Symp_c (P)$. 

Pick $K : \mathbb{A} \to \RR$ supported in a small neighborhood of $\mathcal{D}$ 
such that $K = 1$ in a neighborhood of the closure $\overline{\mathcal{D}}$.  
Denote by $\chi^t$ the time-$t$ map generated by the Hamiltonian flow of $K$. 
$\chi^t$ is supported in a disk in $\mathbb{A}$, hence $\rho (\chi^t) = 0$ for all $t$.

Consider the homomorphism $i_* : H^1_c (P; \RR) \to H^1_c (\mathbb{A}; \RR)$ induced by inclusion $i : P \to \mathbb{A}$. 
Both $\chi^t, \phi_P$ are Hamiltonian in $\mathbb{A}$, hence their fluxes are zero in $H^1_c(\mathbb{A}; \RR)$.
After restriction to $P$, $flux (\chi^t\big|_P), flux (\bar{\phi}_P)$ belong to one-dimensional subspace $\ker i_* \subset H^1_c (P; \RR)$.
$flux (\chi^t\big|_P) \neq 0$, therefore one can find an appropriate $t_\phi \in \RR$
such that $\phi'_P = \chi^{t_\phi} \circ \phi_P$ restricts to $\bar{\phi}'_P = \chi^{t_\phi}\big|_P \circ \bar{\phi}_P$ 
with zero flux in $P$.
The results of ~\cite{Ba:structure-groups} imply $\bar{\phi}'_P \in Ham (P)$.

Pick a compactly supported function $F_t : P \times [0, 1] \to \RR$ whose flow generates $\bar{\phi}'_P$.
Denote by $U_s$ the complement of the closed disk $\overline{j_s(\mathcal{D})}$ in $S^2$, it is a displaceable disk. 
$j_{s,*} (\phi'_P) \in Ham(S^2)$ and it is supported in $U_s$, therefore
\[
	Cal_{j_s} (\phi'_P) = Cal_{U_s} (j_{s,*} (\phi'_P)) = 
				\int_0^1 \left( \int_{U_s} j_{s,*} F_t \omega \right) \dd t = \int_0^1 \left( \int_P F_t \omega \right) \dd t
\]
is independent of $s$. Hence $\rho(\phi'_P) = 0$. From $\phi'_P = \chi^{t_\phi} \circ \phi_P$
we obtain that $$|\rho (\phi'_P) - \rho (\chi^{t_\phi}) - \rho (\phi_P)| = |\rho (\phi_P)|$$
is bounded by the defect of $\rho$.
It follows that $\rho$ is bounded on the subgroup $S'_0$. As $\rho$ is homogeneous, it vanishes there.

\begin{remnonum}
  Choice of different parameters $s_1, s_2$ in ~\eqref{eq:qm_def} gives rise to 
  different quasimorphisms $\rho_A = \rho_{A, s_1, s_2}$. If we fix $s_1$ and let $s_2$ vary, we get a linearly independent family.
  We show that by an argument from ~\cite{Bi-En-Po:Calabi-qm}: 
  let $H (\theta, h) = H (h): \mathbb{A} \to \RR$ be a function which depends only on $h$ and denote by $\phi_H \in \Ham(\mathbb{A})$ 
  the corresponding Hamiltonian.
  The median level set relevant for $Cal_{j_s} (\phi_H)$ is $S^1 \times \{A - s\}$, hence
  \[
	\rho_{A, s_1, s_2} (\phi_H) = Cal_{j_{s_2}} (\phi_H) - Cal_{j_{s_1}} (\phi_H) = 2A \cdot (H (A-s_1) - H (A-s_2)).
  \]
  If we consider the set of all Hamiltonians generated by functions $H = H(h)$, 
  the values of the family $\{\rho_{A, s_1, s_2}\}_{s_2}$ applied to this set are linearly independent.
\end{remnonum}

\section{Discussion} \label{S:disc}

\subsection {Comparison with the Riemannian length spectrum}
In the case of a Riemannian manifold $(M, g)$ the marked length spectrum $l : \pi_1 (M) \to \RR$ contains a lot of information 
regarding the metric $g$. For example, in a closed negatively curved surface $(M, g)$, $g$ is completely determined by $l$ (~\cite{Otal:spec}).

The symplectic analogue looks much less rewarding. The only case when $l_A : H_1 (M; \ZZ) \to \RR$ is not bounded is the case of
a punctured disk. Therefore Hofer's length spectrum is able to detect genus zero surfaces of bounded area. It may happen that more
accurate estimates of $l_A$ provide additional information regarding the topology of $M$. 

On the other side, the symplectic setting is 
less rich than the Riemannian one. The equivalence class of $(M, \omega)$ is determined by the genus, $\text{rk} (H_1 (M; \ZZ))$ and $\Area(M)$.
In the case of finite area $\omega$ is just a scaling factor, hence asymptotic behavior behavior of $l_A$ depends only on the topology of $M$.
Clearly, $l_A$ is able to extract some of the topological information of $M$. This result may be interpreted as the fact that the geometry of $\Ham (M)$ sees some of
the topology of $M$. 

\subsection {Homotopical spectrum} \label{S:disc_homo}
It would be interesting to provide similar estimates for $l_A : \pi_1 (M) \to \RR$. 
When we consider the asymptotical behavior of $l_A$, it is more convenient to discuss the stabilized version 
\[
  \bar{l}_A (\alpha) = \lim_{n \to \infty} \frac{l_A (n \alpha)}{n}.
\]
The homological version of $\bar{l}_A$ is a norm on $H_1$ when $2 A > \Area(M)$ and $\genus(M) = 0$. In all other cases $\bar{l}_A \equiv 0$.

If we replace $H_1$ by $\pi_1$, the argument from the previous sections gives the following partial information:
\begin{itemize}
  \item
	Suppose that $2 A \leq \Area(M)$. Then $\bar{l}_A \equiv 0$. This is true since the argument of Section~\ref{S:up-disp} 
	applies for $\pi_1$ and not just $H_1$.
  \item
	Suppose $2 A > \Area(M)$ and $\genus(M)>0$. Assume that $\alpha$ is represented by a simple nonseparating loop. Then
	$n \alpha$ can be represented by catenation of two simple loops, hence $l_A (n \alpha) \leq 2 A$. Therefore $\bar{l}_A (\alpha) = 0$
	and $\bar{l}_A (k \alpha) = 0$ for all $k \in \ZZ$.
  \item
	Suppose $2 A > \Area(M)$ and $\genus(M)=0$. If $\alpha$ corresponds to a nonzero homology class
	then $\bar{l}_A (\alpha) > 0$.
\end{itemize}

The author does not know estimates for most of the remaining cases. For example, let $2 A > \Area(M)$ and $\genus(M) geq 2$.
For a primitive self intersecting class $\alpha$ the number of simple curves needed to represent $n \alpha$ is not bounded as $n \to \infty$
(see ~\cite{Ca:wlen}), so the argument for the upper bound does not give a lot. On the other hand, we do not know any tools that may give a nontrivial 
lower bound. There are several constructions of quasimorphisms on positive genus surfaces, but none of them is known to respect Hofer's metric. 
When $\genus(M) = 0$ and $\alpha$ belongs to the commutator subgroup $[\pi_1 (M), \pi_1 (M)]$, the argument used in this paper
fails as translation by $\alpha$ cannot be realized by an autonomous flow. It may happen that Entov-Polterovich quasimorphisms
used in our argument still imply a nontrivial lower bound. Unfortunately, the author was not able to compute them on appropriate 
non-autonomous examples.

D. Calegari observed that if one takes the word length norm with 
respect to the generating set of simple loops, then $A$ times its stabilized version  $A \| \cdot \|_S$ satisfies all properties of $\bar{l}_A$ described above
(we assume $2A > \Area(M)$). Indeed, the discussion in Section~\ref{S:up} implies $\bar{l}_A \leq A \| \cdot \|_S$ and we do not know any example 
where the inequality is strict.

\subsection{Higher dimensions}
Most arguments and constructions used in this paper completely fail when one considers balls or tori in higher dimension manifolds.
For example, upper bounds for $l_A (\alpha)$ were induced by decomposition of $\alpha$ into simple loops. 
In dimension $\geq 4$ any loop can be perturbed into a simple one. However, we cannot translate a ball through a pipe containing
such simple loop because of the `symplectic camel' phenomenon.

\subsection{The quasimorphism $\rho_A$} \label{S:disc_quas}
Theorem~\ref{T:quas} describes quasimorphisms $\rho_A : \Ham (M) \to H_1 (M; \RR)$ for a $k$ times punctured disk $M$.
$\rho_A$ can be seen as a generalization of the rotation number $\rho : \widetilde{Homeo}_+ (S^1) \to \RR$. Indeed, $\rho$
satisfies the following: it is Lipschitz with respect to the $C_0$-norm on $\widetilde{Homeo}_+ (S^1)$. If $f \in \widetilde{Homeo}_+ (S^1)$
has a fixed point $x$, then $\rho (f)$ computes the degree of the trajectory of $x$ under $f$. In fact, $\rho$ is determined by the 
second property, while there is a continuous family if linearly independent homogeneous quasimorphisms satisfying same properties as $\rho_A$.

Let $f_t$ be a Hamiltonian isotopy and $\mathcal{D}$ a disk of area $A$ or greater. 
The computation in Section~\ref{S:qm} shows that $\rho_A$ vanishes on those isotopies $f_t$ that fix $\partial \mathcal{D}$ for all $t$.
This implies that the value $\rho_A (f_1)$ is determined up to a bounded defect by the trajectory $f_t (\partial \mathcal{D})$.
That is, for any $g_t$ such that $g_t (\partial \mathcal{D}) = f_t (\partial \mathcal{D})$, $\rho_A (g_1) - \rho_A (f_1)$ is bounded.
Therefore $\rho_A$ can be thought as an invariant of Hamiltonian isotopies of $\partial \mathcal{D}$. 

In particular, $\rho_A$ nearly ignores those dynamical attributes that do not affect a large enough disk in $M$, it sees only `global' features that
deform all large disks. In this sense it is different from most other generalizations of rotation number to dimension two.

A possible application is the ability to detect non-existence of large fixed or periodic disks. Namely, if we pick two
such quasimorphisms $\rho_A, \rho'_B$ and $\rho_A(f) \neq \rho'_B (f)$, then
$f$ has neither fixed nor periodic disks of area $\geq \max\{A, B\}$. More than that, if one wishes to perturb $f$ in order to have such a periodic 
disk, $\rho_A(f) - \rho'_B (f)$ allows to estimate a lower bound on Hofer's energy of such perturbation. However, current techniques 
allow to compute the value of $\rho_A$ in very limited number of scenarios (like the case of autonomous Hamiltonians, locally supported and etc.) 
which limits practical use of this observation.

Another application is as follows. Pick a non-displaceable disk $\mathcal{D}$ in $M$ and two linearly independent quasimorphisms $\rho_A, \rho'_B$ 
where $A, B \leq \Area{\mathcal{D}}$. Denote by $\mathcal{L}$ the space of Lagrangians Hamiltonian isotopic to $\partial \mathcal{D}$.
Suppose that $f_t$ is a Hamiltonian isotopy. A simple computation shows that $\tau = \rho_A - \rho'_B$ depends (up to a bounded defect) just on
$f_1 (\partial \mathcal{D}) \in \mathcal{L}$. This way $\tau$ gives rise to an invariant on $\mathcal{L}$ which is Lipschitz with respect to the induced Hofer's norm. 
$\tau$ is not bounded since it is a nonzero homogeneous quasimorphism, therefore by Lipschitz property the space $\mathcal{L}$ is not bounded in Hofer's norm as well
(a similar argument with more details can be found in ~\cite{Kh:thesis}).

\bibliography{bibliography}

\end{document}